\documentclass[reqno,12pt,letterpaper]{amsart}
\usepackage{amsmath,amssymb,amsthm,graphicx,mathrsfs,url}
\usepackage[usenames,dvipsnames]{color}
\usepackage[colorlinks=true,linkcolor=Red,citecolor=Green]{hyperref}
\usepackage{amsxtra,esint}
\usepackage{mathtools}
\usepackage{MnSymbol}

\usepackage{dsfont}
\usepackage{commath}
\usepackage{comment}
\usepackage{cleveref}
\usepackage{enumerate}

\usepackage{float}
\usepackage[font=small,labelfont=bf]{caption}

\def\?[#1]{\textbf{[#1]}\marginpar{\Large{\textbf{??}}}}

\newtheorem{prop}{Proposition}[section]
\newtheorem{defi}[prop]{Definition}
\newtheorem*{theorem*}{Theorem}
\newtheorem{theorem}[prop]{Theorem}

\newtheorem{rem}[prop]{Remark}

\numberwithin{equation}{section}

\newcommand{\vp}{\varphi}
\newcommand{\R}{\mathbb{R}}

\newcommand{\la}{\langle}
\newcommand{\ra}{\rangle}

\newcommand{\C}{\mathbb{C}}
\renewcommand{\phi}{\varphi}

\renewcommand{\Re}[1]{{\,\mathfrak{Re}} \left ( #1\right ) }
\renewcommand{\Im}[1]{{\,\mathfrak{Im}} \left ( #1 \right ) }

\newcommand{\eps}{\varepsilon}

\DeclareMathOperator{\spn}{span}

\begin{document}

\title[A characterization of complex stable phase retrieval]{A characterization of complex stable phase retrieval in Banach lattices}

\author[Cam\'u\~nez]{Manuel Cam\'u\~nez}
\address{41720 Sevilla, Spain.}
\email{manuelcamunez@gmail.com}

\author[Garc\'ia-S\'anchez]{Enrique Garc\'ia-S\'anchez}
\address{Instituto de Ciencias Matem\'aticas (CSIC-UAM-UC3M-UCM)\\
Consejo Superior de Investigaciones Cient\'ificas\\
C/ Nicol\'as Cabrera, 13--15, Campus de Cantoblanco UAM\\
28049 Madrid, Spain.
\newline
	\href{https://orcid.org/0009-0000-0701-3363}{ORCID: \texttt{0009-0000-0701-3363} } }
\email{enrique.garcia@icmat.es}

\author[de Hevia]{David de Hevia}
\address{Instituto de Ciencias Matem\'aticas (CSIC-UAM-UC3M-UCM)\\
Consejo Superior de Investigaciones Cient\'ificas\\
C/ Nicol\'as Cabrera, 13--15, Campus de Cantoblanco UAM\\
28049 Madrid, Spain.
\newline
\href{https://orcid.org/0009-0003-5545-0789}{ORCID: \texttt{0009-0003-5545-0789}}}
\email{david.dehevia@icmat.es}



\begin{abstract}
This note provides a characterization of the subspaces of a complex Banach lattice which do stable phase retrieval, in the spirit of the characterization of real stable phase retrieval established by D. Freeman, T. Oikhberg, B. Pineau and M. A. Taylor.
\end{abstract}


\maketitle

\section{Introduction}

Phase retrieval problems are of particular interest in applied harmonic analysis due to applications in multiple branches of physics and engineering, such as optics, astronomy, quantum mechanics, speech recognition, and, especially, diffraction imaging (see the surveys \cite{surveyGKR,surveyJEH} and the references therein). In all of these fields, the measuring process of a certain physical magnitude entails a certain loss of information that needs to be reconstructed. More specifically, we want to recover an unknown function from the measurements of a certain positive magnitude, such as its modulus or the modulus of its Fourier or Gabor transform. Note that there will always be a trivial ambiguity when retrieving the phase (or sign) of a function from its modulus, since a function $f$ and $\lambda f$ have the same modulus whenever $\lambda$ is a scalar of modulus one. Hence, the problem is to uniquely determine (up to a trivial ambiguity) the target function from the information given by these positive measurements and any other additional restriction that our target is known to satisfy \textit{a priori}. Moreover, the measurements can be influenced by the experimental procedure, that can introduce small errors. Therefore, it becomes necessary to study the stable version of the problem, that is, quantifying the robustness under small perturbations of the phase retrieval procedure.

These questions have been studied in physics and engineering since at least 1933, when W. Pauli asked whether a wave function can be recovered from its probability densities of position and momentum (equivalently, the modulus of the wave function and its Fourier transform) \cite{Pauli}. In \cite{BCE}, a purely mathematical approach was introduced to tackle phase retrieval problems in frame theory. This initiated an active study of phase recovery in many different settings, such as Fourier phase retrieval or Gabor phase retrieval. Since all these particular instances of the phase retrieval problem involve linear, order theoretic and metric notions, the authors of \cite{FOPT} (see also \cite{CPT,Taylor}) suggested the framework of Banach lattices, a natural generalization of function spaces, as the most convenient abstract setting for dealing with these kinds of questions. 

A (real) \textit{Banach lattice} is a Banach space over the field of real numbers endowed with a partial order that is compatible with the linear structure and admits lattice operations of supremum, infimum and absolute value, in such a way that the norm is compatible with the lattice structure, i.e., $\|x\|\leq \|y\|$ whenever $|x|\leq |y|$. We refer the reader to the classical textbooks on the topic \cite{AliprantisBurkinshaw, LindenstraussTzafririVol2, MeyerNieberg} for the precise definitions, basic notions and standard notation of the theory. 

Banach lattices are a generalization of the usual function spaces such as sequence spaces $\ell_p$, spaces of continuous functions $C(K)$, Lebesgue, Lorentz and Orlicz spaces of functions over measure spaces, etc. However, in a certain way, general Banach lattices are not far from being spaces of functions themselves. In fact, it is well known (see \cite[Lemma 3.4]{Lotz} and \cite{up} for more recent developments) that every Banach lattice can be isometrically embedded into a space of the form $(\bigoplus_{\mu\in \Gamma}L_1(\mu))_\infty$ (i.e., a direct sum of $L_1$-spaces for a certain set of measures $\Gamma$, endowed with the $\ell_\infty$-norm given by the supremum of the norms on each of the factors) in  a way that preserves its lattice structure (the linear operators that preserve the lattice operations are called \textit{lattice homomorphisms}). Moreover, every Banach lattice resembles locally a $C(K)$-space (a space of continuous functions over a compact Hausdorff space $K$): if $X$ is a Banach lattice and $e$ is a non-zero positive element, then $X_e=\{x\in X\::\: |x|\leq \lambda \, e\text{ for some } \lambda>0\}$, \textit{the principal ideal generated by $e$}, can be (lattice isometrically) identified with a $C(K)$-space by means of Kakutani's representation theorem for AM-spaces \cite[Theorem 2.1.3]{MeyerNieberg}.

This similarity with function spaces endows Banach lattices with an interesting feature, usually referred to as \textit{Krivine's functional calculus}. Given a (real) Banach lattice $X$ and some vectors $x_1,\ldots,x_n\in X$, this functional calculus enables us to define in $X$ any expression of the form $h(x_1,\ldots,x_n)$, where $h$ is a continuous and positively homogeneous function on $\R^n$, such as $(\sum_{i=1}^n |x_i|^p)^{\frac{1}{p}}$ for $1< p< \infty$, $|x_1|^{1-\theta}|x_2|^\theta$ for $0<\theta<1$, or in general some functions that cannot be obtained by just using a finite number of linear and lattice operations over the vectors $x_1,\ldots, x_n$. This correspondence allows us to immediately translate well-known equalities and inequalities of real numbers to Banach lattices. Note that, if the Banach lattice $X$ happens to be a space of functions with the usual pointwise or almost everywhere order, the corresponding functional calculus for vectors $f_1,\ldots,f_n\in X$ and a positively homogeneous expression $h:\R^n\rightarrow \R$ consists of the function given by $h(f_1,\ldots,f_n)(t)=h(f_1(t),\ldots,f_n(t))$ for every $t$ in the domain. We will eventually use this tool in order to prove \Cref{thm: complex stable phase retrieval} for general Banach lattices, but the proof can be perfectly followed without any additional knowledge of functional calculus arguments by  assuming that the ambient Banach lattice is just a $C(K)$ or $L_p(\mu)$ space. However, we refer the reader interested in the details of the construction to \cite[Chapter 16]{DiestelJarchowTongue}, \cite[Section 1.d]{LindenstraussTzafririVol2} or \cite[Theorem 2.1.20]{MeyerNieberg}.

Functional calculus is also involved in the definition of \emph{complex Banach lattices}. Let $X$ be a real Banach lattice. The complexification of $X$, denoted by $X_\C$, is the complex vector space $X\times X\equiv X \oplus iX$ endowed with the usual coordinatewise addition and the multiplication by complex scalars given by $(a+ib)(x_1+ix_2)=ax_1-bx_2+i(bx_1+ax_2)$ for $x_1,x_2\in X$, $a,b\in \R$. It is well known that the modulus of a complex-valued function $f=f_1+if_2$ is given by the expression
\[|f|=\sqrt{|f_1|^2+|f_2|^2}=\sup\bigl\{f_1\cos\theta+f_2\sin\theta\::\: \theta\in [0,2\pi]\bigr\}.\]
Analogously, we can use functional calculus to define the modulus of an element $x=x_1+ix_2\in X_\C$ as
\[|x|=\sqrt{|x_1|^2+|x_2|^2}.\]
Note that $|x|$ is an element of $X$, so we can define a norm in the whole complexification $X_\C$ by $\|x\|_{X_\C}=\||x|\|$, in such a way that $(X_\C,\|\cdot\|_{X_\C})$ is a complex Banach space. A \emph{complex Banach lattice} is defined to be the complexification $X_\C$ of a real Banach lattice $X$ endowed with the modulus and the norm defined in this way. Note that Banach spaces of complex-valued functions such as $C(K;\C)$ or $L_p(\mu;\C)$ are particular cases of complex Banach lattices, as they are the complexifications of their real-valued versions $C(K;\R)$ or $L_p(\mu;\R)$. \\

Going back to the main topic of the paper, let us introduce the precise definitions that we will be working with. Given a real or complex Banach lattice $X$ and a (not necessarily closed) subspace $E$ of $X$, we say that $E$ does \emph{phase retrieval (PR)} if for any $f,g\in E$ such that $|f|=|g|$, it follows that $f=\lambda g$ for some scalar $\lambda$ with $|\lambda|=1$. On the other hand, $E$ does \emph{stable phase retrieval with constant $C$} (denoted \emph{$C$-SPR}) if for any $f,g\in E$ the inequality
\begin{equation}\label{eq: SPR}
    \min_{|\lambda|=1}\|f-\lambda g\|\leq C \||f|-|g|\|
\end{equation}
holds. Eventually, we will simply say that $E$ does \emph{stable phase retrieval} (or SPR for short) if there is a constant $C\geq 1$ such that $E$ does stable phase retrieval with constant $C$. Let us define the equivalence relation $\sim$ on $E$ by $f\sim g$ if and only if $f=\lambda g$ for some scalar $\lambda$ of modulus 1 and the map 
\[\fullfunction{\Phi}{E/\sim}{X_+}{[f]}{|f|}.\]
Then, $E$ does PR if and only if the map $\Phi$ is injective. Further, if $E/\sim$ is endowed with the quotient metric, then $E$ does $C$-SPR if and only if the inverse of the map $\Phi$ is $C$-Lipschitz. Note that if $E$ does $C$-SPR, its closure does $C$-SPR too. It is also clear that SPR implies PR.

In the next section, we will review the problem of characterizing phase retrieval and stable phase retrieval in both real and complex Banach lattices, and prove the main result of the paper, \Cref{thm: complex stable phase retrieval}, that provides a \emph{characterization of complex stable phase retrieval}.

\section{Characterizations of stable phase retrieval in Banach lattices}\label{sect: characterization}

One of the main goals of the current research on (stable) phase retrieval is to find characterizations of the phase retrieval and stable phase retrieval conditions that may allow one to prove further results and develop algorithms that can be used in experimental applications. As we mentioned earlier, phase retrieval problems were originally introduced for frames on Hilbert spaces \cite{BCE}, where these questions have been intensively studied. Given a Hilbert space $H$, a \textit{frame} for $H$ is a collection $\Theta=(\vp_j)_{j\in J}\subseteq H$ (with $J$ not necessarily discrete) such that the analysis operator
\[\fullfunction{T_\Theta}{H}{\ell_2(J)}{x}{(\la x,\vp_j\ra)_{j\in J}}\]
is an isomorphic embedding. 

In this setting, a frame $\Theta$ is said to do \textit{phase retrieval (PR)} if whenever $|T_\Theta x|=|T_\Theta y|$ for $x,y\in H$, then $x=\lambda y$ for some unimodular scalar $\lambda$, or equivalently, if the subspace $T_\Theta(H)$ of the Banach lattice $\ell_2(J)$ does PR in the sense defined above. For real scalars, it was established in \cite{BCE} for finite dimensional Hilbert spaces, and later extended to infinite dimensional Hilbert spaces \cite{CCD} and reflexive Banach spaces \cite{AG}, that a frame does PR if and only if it satisfies the complement property. A frame $\Theta$ on a Hilbert spaces $H$ has the \textit{complement property} if for any subset $S\subseteq J$ either $\overline{\spn(\vp_j)_{j\in S}} =H$ or $\overline{\spn(\vp_j)_{j\notin S}} =H$ (see \cite[Definition 2.1]{AG} for the definition in general Banach spaces). In the complex setting, this condition is known to be only necessary \cite{AG,BCE, BCMN, CCD}, and it is asked in \cite[Remark 2.1]{CEHV} whether an analogous characterization exists. Stability of the phase recovery map associated to a frame has also been studied. Since the analysis operator is an embedding, frames doing \textit{stable phase retrieval (SPR)} can be defined as those frames such that the image of the corresponding analysis operator does SPR in $\ell_2(J)$. As was the case with PR, in the real case finite dimensional SPR was shown in \cite{BCMN} to be equivalent to the \textit{strong complement property}, a quantitative version of the complement property, and this characterization was later extended to the infinite dimensional setting \cite{AG}. Again, this condition is only necessary for complex scalars, as it is shown in \cite[Theorem 3.9]{AG}.  

The phase retrieval behavior of frames depends strongly on whether the Hilbert space is finite or infinite dimensional. Indeed, it was proven in \cite{BW, BCMN} (see also \cite{CCD} for the complex case) that in finite dimensions PR and SPR are equivalent. On the other hand, it was established in \cite{CCD}, and later refined in \cite{AG}, that no frame on a infinite dimensional (Hilbert or Banach) space does SPR. Therefore, infinite dimensional subspaces of a Banach lattice doing SPR, which have been recently constructed in \cite{CDFF, CPT, FOPT}, cannot be the range of the analysis operator of a frame.

In other settings, the existence of disjointly supported functions (or similar conditions involving products of functions, see \cite{CCS,CCSW} for some recent examples) seems to be the main obstruction for PR. Actually, it has been shown in \cite{FOPT} that the role that the (strong) complement property plays for frames corresponds in the setting of general Banach lattices to the property of not containing (almost) disjoint pairs: it characterizes real PR (respectively, SPR), and is a necessary condition in the complex case (see Propositions \ref{prop: Phase retrieval real case} and \ref{prop: adp implies no SPR in complex} and \Cref{thm: real stable phase retrieval} below). In this paper, we provide conditions that characterize complex PR and SPR (see \Cref{prop: Phase retrieval complex case} and \Cref{thm: complex stable phase retrieval}). Before stating them, we need to introduce some facts regarding the functional calculus that we recalled in the first section.\\

Let $X$ be a real Banach lattice. Recall that, given $f,g \in X$, Krivine's functional calculus allows us to define the expression $|fg|^{\frac{1}{2}}=|f|^{\frac{1}{2}}|g|^{\frac{1}{2}}$. When $X$ is a Banach lattice of functions with the usual pointwise or almost everywhere order, then $(|fg|^{\frac{1}{2}})(t)=|f(t)g(t)|^{\frac{1}{2}}$ for every $t$ in the domain. Moreover, the equality
\begin{equation}\label{eq: supremum times infimum}
    |fg|^{\frac{1}{2}}=(|f|\vee|g|)^{\frac{1}{2}}(|f|\wedge|g|)^{\frac{1}{2}}
\end{equation}
holds in $\R$, so it can be transferred to $X$.

Similarly, if $X$ is a complex Banach lattice, let $\hat{X}$ be the real Banach lattice such that $X=\hat{X}+i\hat{X}$ and let $f=f_1+if_2, g=g_1+ig_2\in X$ with $f_j, g_j\in \hat{X}$, $j=1,2$. We can define the conjugate of $g$ as $\overline{g}=g_1-ig_2$. Then, if $X$ was a Banach lattice of functions, the meaning of the expression $|\Re{f\overline{g}}|^{\frac{1}{2}}\in \hat{X}$ would be clear. Since $X$ is now a general complex Banach lattice, we need to define this expression in $\hat{X}$ by means of functional calculus as
\[|\Re{f\overline{g}}|^{\frac{1}{2}}=|f_1g_1+f_2g_2|^{\frac{1}{2}}=|\Re{f}\Re{g}+\Im{f}\Im{g}|^{\frac{1}{2}}.\]
Note that the equalities
\begin{equation}\label{eq: homogeneity of real part}
    \bigl|\Re{(\lambda f)\overline{(\mu g)}}\bigr|^{\frac{1}{2}}=|\lambda|^{\frac{1}{2}}|\mu|^{\frac{1}{2}}|\Re{f\overline{g}}|^{\frac{1}{2}}
\end{equation}
and
\begin{equation}\label{eq: real part and sum times difference}
    2|\Re{f\overline{g}}|^\frac{1}{2}=\left||f+g|^2-|f-g|^2\right|^\frac{1}{2}=\left||f+g|+|f-g| \right|^\frac{1}{2} \left||f+g|-|f-g| \right|^\frac{1}{2}
\end{equation}
are true in $\R$, provided $\lambda,\mu \in \R$. Therefore, they also hold in $\hat{X}$.\\

Now that we have extended the above expressions to the setting of Banach lattices, let us go back to the original question of characterizing when a subspace of a Banach lattice does phase retrieval. As we advanced earlier, in the real setting, the main and only obstruction to phase retrieval turns out to be the existence of disjoint pairs of elements in our subspace.

\begin{prop}[Real phase retrieval]\label{prop: Phase retrieval real case}
    Let $E$ be a subspace of a real Banach lattice $X$. The following are equivalent:
    \begin{enumerate}
        \item $E$ fails PR.
        \item There exist non-zero $f,g\in E$ such that $|f|\land |g|=0$.
        \item There exist non-zero $f,g\in E$ such that $|fg|^{\frac{1}{2}}=0$.
    \end{enumerate}
\end{prop}
\begin{proof}
 $(1)\Rightarrow (2)$ Suppose that $E$ fails PR. Then, there exist $f,g\in E$ with $|f|=|g|$ and $f\neq \pm g$. In particular, $f-g$ and $f+g$ are non-zero elements of $E$, and using some of the identities stated in \cite[Theorem 1.1.1]{MeyerNieberg} we deduce that
 \begin{align*}
     |f+g|\land |f-g|  = & \bigl((f+g)\lor (-f-g) \bigr) \land \bigl((f-g)\lor (-f+g) \bigr) \\
     =&\bigl((f+g)\land (f-g) \bigr)\lor \bigl((f+g)\land (-f+g) \bigr)  \\
     & \lor  \bigl((-f-g)\land (f-g) \bigr)\lor \bigl((-f-g)\land (-f+g) \bigr)  \\
     =& (f-|g|)\lor (-|f|+g) \lor (-|f|-g)\lor (-f-|g|) \\
     =&\bigl((f-|g|)\lor (-f-|g|) \bigr)\lor \bigl((-|f|+g) \lor (-|f|-g) \bigr) \\
     =& (|f|-|g|)\lor (-|f|+|g|)= \bigl||f|-|g| \bigr|=0.
 \end{align*}

 \noindent$(2)\Rightarrow(1)$ Let $f,g\in E$ be two disjoint vectors. Then, in particular, they are linearly independent and so are $f+g$ and $f-g$. Now taking into account that $f$ and $g$ are disjoint, by \cite[Theorem 1.1.1 (vi)]{MeyerNieberg},  $|f+g|=|f|+|g|=|f-g|$, so $E$ fails PR.\\

  \noindent$(2)\Leftrightarrow(3)$ Let $f,g\in E$, and consider the principal ideal generated by $h=|f|+|g|$ in $X$, $X_h$, which can be identified with some $C(K)\equiv C(K;\R)$ by means of a bijective lattice homomorphism $J:X_h\rightarrow C(K)$ \cite[Theorem 2.1.3]{MeyerNieberg}. On one hand, if $f$ and $g$ are disjoint in $X$, so are $Jf$ and $Jg$ in $C(K)$, where being disjoint means having disjoint supports. Hence, $Jf\cdot Jg=0$ and, in particular, $|Jf\cdot Jg|^\frac{1}{2}=0$. Since $J$ is bijective and the functional calculus is preserved by lattice homomorphisms, we recover that $|f g|^\frac{1}{2}=0$. Similarly, if $|f g|^\frac{1}{2}=0$, the uniqueness of the functional calculus in $C(K)$ implies that $|Jf\cdot Jg|^\frac{1}{2}=0$, so $Jf$ and $Jg$ must have disjoint supports and thus $f$ and $g$ are disjoint in $X$. 
\end{proof}

Recall that a frame $\Theta=(\varphi_j)_{j\in J}$ for a Hilbert space $H$ does PR if and only if the range of the analysis operator, $T_\Theta(H)$, does PR in $\ell_2(J)$, and by \cite{BCE, CCD} this was characterized in terms of the complement property. The connection between this property and $T_\Theta(H)$ not containing disjoint elements is clear: Indeed, if $\Theta$ fails the complement property, there exists a set $\varnothing\neq S\subsetneq J$ such that $F_1=\overline{\spn(\vp_j)_{j\in S}} \subsetneq H$ and $F_2=\overline{\spn(\vp_j)_{j\notin S}} \subsetneq H$, so we can find non-zero vectors $u\in F_1^\perp$ and $v\in F_2^\perp$ (here $\perp$ denotes the orthogonal complement of a subspace of $H$). It follows that $T_\Theta u$ and $T_\Theta v$ have disjoint support in $\ell_2(J)$, so $T_\Theta (H)$ contains disjoint pairs. Conversely, if $T_\Theta (H)$ contains a disjoint pair $T_\Theta u$ and $T_\Theta v$, we can consider the set $S=\text{supp}(T_\Theta u)$, so that $\text{supp}(T_\Theta v)\subseteq J\setminus S$, and therefore
\[v\in \intoo[1]{\overline{\spn(\vp_j)_{j\in S}}}^\perp \quad \text{and} \quad u\in \intoo[1]{\overline{\spn(\vp_j)_{j\notin S}}}^\perp ,\]
so $\Theta$ also fails the complement property.\\

As is the case with frames, studying complex PR for Banach lattices is slightly more involved, as disjointness is not the only obstruction for phase retrieval. Indeed, if a subspace $E$ of a complex Banach lattice $X=\hat{X}+i\hat{X}$ contains two real and linearly independent elements $f,g\in E\cap \hat{X}$, then $E$ automatically fails phase retrieval, as it contains the elements $f+ig$ and $f-ig$, which are linearly independent but have the same modulus. However, $E$ might not contain disjoint pairs (take for instance $E=\spn \{f,g\}$ with $f$ and $g$ non-disjoint satisfying the previous conditions). In view of this, it seems that it is not condition $(2)$ from \Cref{prop: Phase retrieval real case} but rather $(3)$ that should be used in order to get a generalization to the complex setting.

\begin{prop}[Complex phase retrieval]\label{prop: Phase retrieval complex case}
    Let $E$ be a subspace of a complex Banach lattice $X$. The following are equivalent:
    \begin{enumerate}
        \item $E$ fails PR.
        \item There exist linearly independent $f,g\in E$ such that $|f-g|=|f+g|$.
        \item There exist linearly independent $f,g\in E$ such that $|\Re{f\overline{g}}|^\frac{1}{2}=0$.
        \item There exist linearly independent $f,g\in E$ such that $|f+g|=(|f|^2+|g|^2)^\frac{1}{2}$.
    \end{enumerate}
\end{prop}

\begin{proof}
    $(1) \Leftrightarrow (2)$ $E$ fails PR if and only if there exist $u,v\in E$ linearly independent such that $\abs{u}=\abs{v}$. Considering the linearly independent combinations $f=\frac{u+v}{2}$ and $g=\frac{u-v}{2}$ we can write $u=f+g$ and $v=f-g$, so conditions $(1)$ and $(2)$ are equivalent.\\

    \noindent$(2) \Leftrightarrow (3) \Leftrightarrow (4)$ Assume that $X=\hat{X}+i\hat{X}$ and let $f,g\in E$, $h=|f|+|g|\in \hat{X}$ and $\hat{X}_h$ be the principal ideal generated by $h$. Then, by Kakutani's representation Theorem \cite[Theorem 2.1.3]{MeyerNieberg} we can identify $\hat{X}_h$ with some real $C(K;\R)$, as we did in \Cref{prop: Phase retrieval real case}, and $X_h=\hat{X}_h+i\hat{X}_h$ with $C(K;\C)$. Therefore, we can assume without loss of generality that $f,g\in C(K;\C)$. It is easy to check that
    \begin{equation*}
        \abs{f+g}^2=(f+g)(\bar{f}+\bar{g})=\abs{f}^2+\abs{g}^2+2\Re{f\overline{g}}
    \end{equation*}
    and
    \begin{equation*}
        \abs{f-g}^2=(f-g)(\bar{f}-\bar{g})=\abs{f}^2+\abs{g}^2-2\Re{f\overline{g}},
    \end{equation*}
    so $(2)$, $(3)$ and $(4)$ are clearly equivalent on $C(K;\C)$. The uniqueness of functional calculus allows us to bring this equivalence back to the complex Banach lattice $X$.
\end{proof}

To illustrate the intuition behind condition (3), let $a$ and $b$ be complex numbers such that $\mathfrak{Re}\,\bigl(a\overline{b}\bigr)=0$. This implies that $\arg(a)=\arg(b)\pm \frac{\pi}{2}$, that is, they are perpendicular when viewed as vectors in $\C=\R^2$. Similarly, if $f$ and $g$ are functions on a domain $\Omega$ with complex values, the condition $\Re{f\overline{g}}=0$ implies that at each point $\omega\in \Omega$, the values $f(\omega)$ and $g(\omega)$ are perpendicular in $\C=\R^2$. In analogy with this, the pairs $f,g$ in a Banach lattice $X$ satisfying $|\Re{f\overline{g}}|^\frac{1}{2}=0$ will be called \textit{perpendicular pairs}.

Note that phase retrieval is a purely algebraic property, so Propositions \ref{prop: Phase retrieval real case} and \ref{prop: Phase retrieval complex case} could be stated for vector lattices, as the norm does not play any role. In contrast, for stable phase retrieval the situation is not as straightforward. The authors of \cite{FOPT} provided a full characterization of real stable phase retrieval by changing in \Cref{prop: Phase retrieval real case} the notion of disjointness by that of almost disjointness. Given $\eps>0$, we say that $f,g\in S_E$ is a (normalized) \emph{$\eps$-almost disjoint pair} if $\||f|\wedge|g|\|\leq \eps$. Almost disjoint pairs were used to stablish a quantitative characterization of real SPR in \cite[Theorem 3.4 and Remark 3.5]{FOPT} (see the equivalence $(1)\Leftrightarrow (2)$ and the additional statement in the theorem below), which was later refined in \cite[Proposition 3.4]{Eugene}. It turns out that an equivalent condition can be obtained looking at products (in the functional calculus sense) instead of infima.

\begin{theorem}[Real stable phase retrieval] \label{thm: real stable phase retrieval}
    Let $E$ be a subspace of a real Banach lattice $X$. The following are equivalent:
    \begin{enumerate}
        \item $E$ fails SPR.
        \item For every $\eps>0$, $E$ contains an $\eps$-almost disjoint pair.
        \item For every $\eps>0$, there exist $f,g\in S_E$ such that $\||fg|^\frac{1}{2}\|<\eps$.
    \end{enumerate}
    Moreover, $E$ does $C$-SPR if and only if $E$ does not contain any $\frac{1}{C}$-almost disjoint pair.
\end{theorem}

\begin{proof}
    $(1) \Leftrightarrow (2)$ As we already mentioned, the last part of the statement was established in \cite[Theorem 3.4 and Remark 3.5]{FOPT} and \cite[Proposition 3.4]{Eugene}. In particular, this yields the equivalence between $(1)$ and $(2)$. \\

    \noindent$(2) \Rightarrow (3)$ For every $\eps>0$, take $f,g\in S_E$ such that $\|\abs{f}\wedge\abs{g} \|<\eps$. By equation \eqref{eq: supremum times infimum} and \cite[Proposition 1.d.2(i)]{LindenstraussTzafririVol2} we get that
    \[\||fg|^\frac{1}{2}\|\leq \|\abs{f}\vee\abs{g}\|^\frac{1}{2}\|\abs{f}\wedge\abs{g}\|^\frac{1}{2} \leq\footnote{Here we are using that $\||f|\lor|g|\|\leq \|f\|+\|g\|\leq 2$. Note that if $X$ is an AM-space, i.e., a closed sublattice of a $C(K)$-space (see, e.g., \cite[Theorem 1.b.6]{LindenstraussTzafririVol2}), then the factor $\sqrt{2}$ can be removed in this inequality.} \sqrt{2\eps}.\]

    \noindent$(3) \Rightarrow (2)$ It is clear from equation \eqref{eq: supremum times infimum} that $|f|\wedge |g|\leq |fg|^\frac{1}{2}$, so 
    \[\||f|\wedge |g|\|\leq \||fg|^\frac{1}{2}\|\leq \eps. \qedhere\]
\end{proof}

Note that $(2)\Rightarrow (1)$ holds in the complex setting as well. More specifically, if a subspace $E$ of a complex Banach lattice contains a normalized $\eps$-almost disjoint pair, then it fails SPR with constant $\frac{1}{\sqrt{2}\eps}$. This is probably a known fact in the community, but the proof is not as straightforward as in the real case. For the sake of completeness, we reproduce in \Cref{prop: adp implies no SPR in complex} an argument provided by Jesús Illescas-Fiorito in his Master's thesis \cite[Remark II.3.3]{Illescas}. The proof is based on an ``orthogonal reduction'' argument that was first established in \cite[Theorem 1.1]{AAFG} for frames on Hilbert spaces. This property was later generalized to subspaces of Banach lattices in \cite{FOPT}. Let us recall the statement of this last version:

\begin{theorem}{\cite[Theorem 3.10]{FOPT}}\label{thm: orthogonal reduction}
    Let $(X,\|\cdot\|)$  be a (real or complex) Banach lattice. Fix linearly independent $f,g\in X$, and suppose that $Y=\spn\{f,g\}$ is equipped with a Hilbert space norm $\|\cdot\|_H$ which is $K$-equivalent to $\|\cdot\|$, with $1\leq K \leq \sqrt{2}$. Then there exist $f',g'\in Y$ so that
    \[\min_{|\lambda|=1}\|f-\lambda g\|\leq K \min_{|\lambda|=1}\|f'-\lambda g'\|,\]
    and
    \[\bigl(\|f'\|^2+\|g'\|^2\bigr)^\frac{1}{2}\leq K \min_{|\lambda|=1}\|f'-\lambda g'\|,\]
    and
    \[\left||f'|-|g'|\right|\leq \left||f|-|g|\right|.\]
\end{theorem}

\Cref{thm: orthogonal reduction} establishes that SPR is witnessed on ``almost orthogonal'' vectors, meaning that if $f$ and $g$ satisfy the SPR inequality \eqref{eq: SPR}, then we can find $f'$ and $g'$ in $\spn\{f,g\}$ that tighten the inequality (up to a constant $K\leq \sqrt{2}$) and moreover $\min_{|\lambda|=1}\|f'-\lambda g'\|$ is bounded below by $K^{-1}(\|f'\|^2+\|g'\|^2)^\frac{1}{2}$. In particular, this allows us to assume that both terms in the SPR inequality are large. Moreover, in the proof of \cite[Theorem 3.10]{FOPT} we can find a method for obtaining such $f'$ and $g'$ from $f$ and $g$ that we will use to show:

\begin{prop}\label{prop: adp implies no SPR in complex}
    Let $E$ be a subspace of a complex Banach lattice $X$ and $\eps>0$. If there exist $u,v\in S_E$ such that $\||u|\wedge|v|\|<\eps$, then $E$ fails $\frac{1}{\sqrt{2}\eps}$-SPR.
\end{prop}

\begin{proof}
    Given $u,v\in S_E$ such that $\||u|\wedge|v|\|<\eps$, fix a Hilbert norm $\|\cdot\|_H$ on $\spn\{u,v\}$ induced by a scalar product $\langle\cdot,\cdot \rangle$ such that $\|\cdot\|\leq \|\cdot\|_H\leq \sqrt{2}\|\cdot\|$,
    and assume without loss of generality that $\|u\|_H\geq \|v\|_H$. Let $\mu = \frac{\langle u,v \rangle}{|\langle u,v \rangle|}$ (if $\langle u,v \rangle=0$, take $\mu=1$ instead) and denote $f=u+\mu v$ and $g=u-\mu v$. Observe that $\langle f,g \rangle =\|u\|_H^2- \|v\|_H^2\geq 0$, which implies that
    \[\min_{|\lambda|=1}\|f-\lambda g\|_H=\|f-g\|_H.\]
    Reproducing the proof of \Cref{thm: orthogonal reduction}, there exists some $R\in [0,\frac{1}{2}]$ such that $f'=f-R(f+g)$ and $g'=g-R(f+g)$ satisfy that
    \[\left||f'|-|g'|\right|\leq \left||f|-|g|\right|\]
    and $\langle f',g' \rangle =0$, so
    \[\min_{|\lambda|=1}\|f'-\lambda g'\|_H=\|f'- g'\|_H.\]
    In particular, the fact that $f'-g'=f-g=2\mu v$ implies that 
    \[\min_{|\lambda|=1}\|f'-\lambda g'\|\geq \frac{1}{\sqrt{2}} \min_{|\lambda|=1}\|f'-\lambda g'\|_H\geq \frac{1}{\sqrt{2}} \|f'- g'\|_H =\frac{2}{\sqrt{2}}\|v\|_H \geq \frac{2}{\sqrt{2}}\|v\|= \frac{2}{\sqrt{2}}.\]
    On the other hand, observe that $|f+g|=2\,|u|$ and $|f-g|=2\,|v|$, so
    \[\left||f|-|g|\right|\leq |f+g|\land |f-g|\leq2\, |u|\wedge |v|.\]
    Putting everything together, we conclude that
    \[\||f'|-|g'|\|\leq \||f|-|g|\|< 2\eps \leq \sqrt{2} \eps \min_{|\lambda|=1}\|f'-\lambda g'\|,\]
    so $E$ fails $\frac{1}{\sqrt{2}\eps}$-SPR.
\end{proof}

Using again the orthogonal reduction argument provided in \Cref{thm: orthogonal reduction}, we can obtain the main result of the paper: a characterization of stable phase retrieval in the complex setting. In a similar way to what happens in the real case, where the restriction of not containing disjoint pairs of vectors in \Cref{prop: Phase retrieval real case} became the property of not containing almost disjoint pairs, in the complex, case condition (3) in \Cref{prop: Phase retrieval complex case} will also have a \textit{stable analogue}:

\begin{defi}\label{defi: almost perpendicular pairs}
    Let $X$ be a complex Banach lattice and $\eps>0$. We say that two elements $u,v\in X$ form a \textbf{(normalized) $\eps$-almost perpendicular pair} if $\|u\|=\|v\|=1$ and $\||\Re{u\overline{v}}|^\frac{1}{2} \|<\eps$.
\end{defi}

\Cref{thm: complex stable phase retrieval} states that a subspace of a complex Banach lattice fails SPR if and only if it contains uniformly separated $\eps$-almost perpendicular pairs for every $\eps>0$.

\begin{theorem}[Complex stable phase retrieval]\label{thm: complex stable phase retrieval}
    Let $E$ be a subspace of a complex Banach lattice $X$. The following assertions are equivalent:
    \begin{enumerate}
        \item $E$ fails SPR.
        \item For every $0<m<\frac{1}{\sqrt{2}}-\frac{1}{2}$ and $\eps>0$, there exist $u,v\in S_E$ such that $\min_{\abs{\lambda}=1} \| u-\lambda v\| \geq m$ and $\||\Re{u\overline{v}}|^\frac{1}{2} \|<\eps$.
        \item There exists $m>0$ satisfying that for every $\eps>0$, there exist $u,v\in S_E$ such that $\min_{\abs{\lambda}=1} \| u-\lambda v\| \geq m$ and $\||\Re{u\overline{v}}|^\frac{1}{2} \|<\eps$.
    \end{enumerate}
\end{theorem}

\begin{proof}
    $(1) \Rightarrow (2)$ Let us fix $0<m<\frac{1}{\sqrt{2}}-\frac{1}{2}$ and $\eps>0$. $E$ fails SPR, so given any $C\geq 1$ there exist linearly independent $f,g\in E$ such that 
    \begin{equation*}
        \min_{\abs{\lambda}=1} \| f-\lambda g\| > C \|\abs{f}-\abs{g}\|.
    \end{equation*}
    Following the proof of \Cref{thm: orthogonal reduction} we can find a pair of linearly independent vectors $f',g'\in Y= \spn \{f,g\}$ such that $\|g'\|\leq\| f'\|=1$,
    \begin{equation*}
        \min_{\abs{\lambda}=1} \| f'-\lambda g'\| > \frac{C}{\sqrt{2}} \norm[1]{\,\abs{f'}-\abs{g'}}
    \end{equation*}
    and 
    \begin{equation*}
        \min_{\abs{\lambda}=1} \| f'-\lambda g'\| \geq \frac{1}{\sqrt{2}}\left(\| f'\|^2+\| g'\|^2 \right)^{\frac{1}{2}}.
    \end{equation*}
    Moreover, there exists an Euclidean norm $\|\cdot \|_H$ on $Y$ such that $\|h\|\leq \|h\|_H\leq \sqrt{2}\|h\|$ for every $h\in Y$ and $\langle{f'},{g'}\rangle=0$. Note that, in particular,
    \begin{equation*}
        0 \leq \| f'\|-\| g'\| \leq \norm[1]{\,\abs{f'}-\abs{g'}} \leq \frac{\sqrt{2}}{C} \min_{\abs{\lambda}=1} \| f'-\lambda g'\| \leq \frac{2\sqrt{2}}{C}.
    \end{equation*}
    Let $0<\theta<1$ be a parameter depending only $m$ whose value will be established later in the proof, and assume that $C> \frac{2\sqrt{2}}{\theta}$. Then, 
    \begin{equation*}
        1=\| f'\|\geq \|g'\|\geq \| f'\|- \frac{2\sqrt{2}}{C}\geq 1-\theta
    \end{equation*}
    and
    \begin{equation*}
        \min_{\abs{\lambda}=1} \| f'-\lambda g'\| \geq \frac{1}{\sqrt{2}}\left(1+(1-\theta)^2 \right)^{\frac{1}{2}}.
    \end{equation*}
    Now, let us define $u'=\frac{f'+g'}{2}$ and $v'=\frac{f'-g'}{2}$. Clearly, they are linearly independent, since they generate $f'=u'+v'$ and $g'=u'-v'$. Moreover, by equation \eqref{eq: real part and sum times difference} and \cite[Proposition 1.d.2(i)]{LindenstraussTzafririVol2} it follows that
    \begin{equation}\label{eq: complex SPR key inequality}
        \norm[1]{|\mathfrak{Re}(u'\overline{v'})|^\frac{1}{2}}=\frac{1}{2}\bigl\|\left||f'|+|g'| \right|^\frac{1}{2} \left||f'|-|g'| \right|^\frac{1}{2}\bigr\| \leq \frac{1}{2} \||f'|+|g'| \|^\frac{1}{2} \||f'|-|g'| \|^\frac{1}{2}\leq  \frac{2^\frac{1}{4}}{\sqrt{C}}.
    \end{equation}
    On the other hand, for every unimodular $\lambda\in \C$ we have
    \begin{align*}
        2\|u'-\lambda v'\|^2 & \geq \|u'-\lambda v'\|_H^2  \\
        &= \frac{1}{4}\|(1-\lambda)f'+(1+\lambda) g'\|_H^2 \\
        &= \frac{1}{4} \left(\abs{1-\lambda}^2\|f'\|_H^2 + \abs{1+\lambda}^2\|g'\|_H^2 \right)\\
        & = \frac{1}{2} \left((1-\mathfrak{Re}{\lambda})\|f'\|_H^2 + (1+\mathfrak{Re}{\lambda})\|g'\|_H^2 \right)\\
        & \geq \min \{\|f'\|_H^2,\|g'\|_H^2\}\geq \min \{\|f'\|^2,\|g'\|^2\} = \|g'\|^2 \geq (1-\theta)^2,
    \end{align*}
    and thus
    \begin{equation*}
        \min_{\abs{\lambda}=1} \| u'-\lambda v'\| \geq \frac{1-\theta}{\sqrt{2}}.
    \end{equation*}
    Finally, note that
    \begin{equation*}
         \frac{1}{2\sqrt{2}}\left(1+(1-\theta)^2 \right)^{\frac{1}{2}} \leq  \frac{1}{2}\min_{\abs{\lambda}=1}\|f'-\lambda g'\|\leq \|u'\|\leq\frac{1}{2}(\|f'\|+\|g'\|)\leq 1,
    \end{equation*}
    and the same holds for $\|v'\|$. Define $u=\frac{u'}{\|u'\|}$ and $v=\frac{v'}{\|v'\|}$. Then, $u$ and $v$ are normalized vectors, and they satisfy that
    \begin{align*}
        \|u-\lambda v\|& \geq \frac{1}{\|u'\|}\left( \|u'-\lambda v'\|-\left\|\lambda\left(1- \frac{\|u'\|}{\|v'\|}\right)v'\right\| \right)\\
        & \geq \min_{\abs{\lambda}=1} \| u'-\lambda v'\| -\abs{\|u'\|-\|v'\|} \\
        & \geq \frac{1-\theta}{\sqrt{2}} + \frac{1}{2\sqrt{2}}\left(1+(1-\theta)^2 \right)^{\frac{1}{2}} -1
    \end{align*}
    for every unimodular scalar $\lambda$. Choosing $\theta\in (0,1)$ so that 
    \[m= \frac{1-\theta}{\sqrt{2}} + \frac{1}{2\sqrt{2}}\left(1+(1-\theta)^2 \right)^{\frac{1}{2}} -1\]
    (this is always possible, since $0<m<\frac{1}{\sqrt{2}}-\frac{1}{2}$), and making
    \begin{equation*}
        C>\max\left \{\frac{8\sqrt{2}}{(1+(1-\theta)^2)\eps^2}, \frac{2\sqrt{2}}{\theta}\right \},
    \end{equation*}
     we can use \eqref{eq: homogeneity of real part} and \eqref{eq: complex SPR key inequality} to conclude that 
    \begin{equation*}
        \||\Re{u\overline{v}}|^\frac{1}{2} \|= \frac{1}{\|u'\|^\frac{1}{2}\|v'\|^\frac{1}{2}}\||\Re{u'\overline{v'}} |^\frac{1}{2}\| \leq \left(\frac{8\sqrt{2}}{(1+(1-\theta)^2)C} \right)^\frac{1}{2}<\eps.
    \end{equation*}

    \noindent$(2) \Rightarrow (3)$ is trivial.\\

    \noindent$(3) \Rightarrow (1)$ We want to show that the SPR condition fails in $E$ for any constant $C>0$, that is, for every $C$ there is a pair of linearly independent elements $f,g\in E$ such that 
    \begin{equation*}
        \min_{\abs{\lambda}=1} \| f-\lambda g\| > C \|\abs{f}-\abs{g}\|.
    \end{equation*}
    Let us fix a constant $C>0$ and find $m>0$ satisfying the hypothesis. Then, we take $0<\eps< \frac{\delta m}{2C}$, where $\delta>0$ is a parameter depending only on $m$ that will be fixed later, and $u,v\in E$ as in the statement. Let $f=u+v$ and $g=u-v$, which are clearly linearly independent. It follows using functional calculus and \eqref{eq: real part and sum times difference} that
    \begin{equation}\label{eq: bound 2 implies 1}
        \left||f|-|g|\right|\leq \left||f|-|g|\right|^\frac{1}{2}\left||f|+|g|\right|^\frac{1}{2}= 2\abs{\Re{u\overline{v}}}^\frac{1}{2},
    \end{equation}
    so $\|\abs{f}-\abs{g}\|\leq 2 \||\Re{u\overline{v}}|^{\frac{1}{2}} \|\leq 2\eps$.

    Now, let us provide a lower bound for the expressions $\|f-\lambda g\|$, where $\lambda$ is a unimodular scalar. To do so, note that $\abs{1-\lambda}^2 + \abs{1+\lambda}^2=4$, so
    \begin{equation*}
        \|f-\lambda g\| = \|(1-\lambda)u+(1+\lambda) v\|  \geq 2 M,
    \end{equation*}
    where 
    \begin{equation*}
        M= \inf \left\{ \left\|\alpha u+ \beta v\right\|: \abs{\alpha}^2+\abs{\beta}^2=1 \right\}.
    \end{equation*}
    In order to bound $M$ from below, we select $\delta =(1+(1+\frac{m}{2})^2)^{-\frac{1}{2}}>0$, so that $\frac{\sqrt{1-\delta^2}}{\delta}=1+\frac{m}{2}$, and we distinguish three cases:
    \begin{itemize}
        \item If $\abs{\alpha}\geq \sqrt{1-\delta^2}$ (equivalently, $\abs{\beta}\leq \delta$), then
        \begin{equation*}
            \left\|\alpha u+ \beta v\right\|\geq \|\alpha u\|-\|\beta v\|= \abs{\alpha}-\abs{\beta}\geq \sqrt{1-\delta^2}-\delta = \frac{\delta m}{2}.
        \end{equation*}
        \item If $\abs{\alpha}\leq \delta$ (equivalently, $\abs{\beta}\geq \sqrt{1-\delta^2}$), then
        \begin{equation*}
            \left\|\alpha u+ \beta v\right\|\geq \|\beta v\|-\|\alpha u\|= \abs{\beta}-\abs{\alpha}\geq \sqrt{1-\delta^2}-\delta = \frac{\delta m}{2}.
        \end{equation*}
        \item If $\delta \leq\abs{\alpha}\leq\sqrt{1-\delta^2}$ (equivalently, $\delta \leq\abs{\beta}\leq\sqrt{1-\delta^2}$), then $\frac{\delta}{\sqrt{1-\delta^2}}\leq r=\frac{\abs{\beta}}{\abs{\alpha}}\leq \frac{\sqrt{1-\delta^2}}{\delta}$ and $\abs{r-1}\leq \frac{m}{2}$. Writing $\frac{\beta}{\alpha}=re^{i\varphi}$ we get
        \begin{equation*}
            \left\|\alpha u+ \beta v\right\| = \abs{\alpha}\|u+e^{i\varphi} v + (r-1)e^{i\varphi} v\| \geq \abs{\alpha}\left(\|u+e^{i\varphi} v\| - \abs{r-1} \right)\geq \frac{\delta m}{2}.
        \end{equation*}
    \end{itemize}
    We conclude that $M\geq \frac{\delta m}{2}$, so we have found a pair of linearly independent vectors $f$ and $g$ in $E$ such that
    \begin{equation*}
         C \|\abs{f}-\abs{g}\| \leq 2\eps C < \delta m \leq 2M\leq \min_{\abs{\lambda}=1} \| f-\lambda g\|
    \end{equation*}
    and the proof is concluded.
\end{proof}

\begin{rem}
    Note that the restriction $\min_{\abs{\lambda}=1} \| u-\lambda v\| \geq m >0$ cannot be dropped in conditions $(2)$ and $(3)$ of \Cref{thm: complex stable phase retrieval}. Indeed, there exists a two dimensional subspace of $\C^4$ that does stable phase retrieval, but for every $\eps>0$ it contains a pair of linearly independent vectors $u,v\in S_E$ such that $\||\Re{u\overline{v}}|^\frac{1}{2} \|<\eps$.
\end{rem}

\begin{proof}
    Let $u=(1,1,1,0), w=(1,i,0,1)\in \ell_\infty^4$ (over $\C$) and $E=\spn\{u,w\}$. Given $0<\delta<1$, let $v=A(iu+\delta w)=A(i+\delta, (1+\delta)i,i,\delta)$, where $A=(1+\delta)^{-1}$. It is clear that, for every $\delta$, $u$ and $v$ are normalized and linearly independent. Moreover, $|\Re{u\overline{v}}|^\frac{1}{2}=(\sqrt{A\delta},0,0,0)$, so for every $\eps>0$ we can choose $\delta$ such that $1-A=A\delta=\frac{\delta}{1+\delta}<\eps^2$, so that $\||\Re{u\overline{v}}|^\frac{1}{2} \|_\infty<\eps$. However, we can observe that
    \[\min_{\abs{\lambda}=1} \| u-\lambda v\|_\infty\leq \biggl\| u-\frac{i+1}{i-1} v\biggr\|_\infty=\frac{1}{\sqrt{2}}\bigl\|(i-1)(1-A)u-(i+1)A\delta w\bigr\|_\infty \leq 2A\delta<2\eps^2, \]
    so when $\eps$ is small the pair $u,v$ does not satisfy the separation restriction of \Cref{thm: complex stable phase retrieval}.
    
    Next, let us show that $E$ does SPR. To do so, it suffices to prove that $E$ does PR, since by \cite[Corollary 3.14]{FOPT} PR and SPR are equivalent for finite dimensional subspaces. By \Cref{prop: Phase retrieval complex case}, we need to check that for every pair of linearly independent vectors $f,g\in E$, the expression $|\Re{f\overline{g}}|^\frac{1}{2}$ is non-zero. Assume the contrary: there exist two linearly independent vectors $f=a_1u+b_1w,\: g=a_2u+b_2w\in E$ such that $|\Re{f\overline{g}}|^\frac{1}{2}=0$. Since $\ell_\infty^4$ is a space of functions, the functional calculus coincides with the usual functional calculus in $\R$ coordinatewise, so, in particular, we can write $\Re{f\overline{g}}=0$. A simple computation shows that
    \begin{equation*}
        \left\{\begin{aligned}
            0=  & (\Re{ a_1}+\Re{ b_1})(\Re{ a_2}+\Re{ b_2})+(\Im{ a_1}+\Im{ b_1})(\Im{ a_2}+\Im{ b_2}),\\
            0= & (\Re{ a_1}-\Im{ b_1})(\Re{ a_2}-\Im{ b_2})+(\Im{ a_1}+\Re{ b_1})(\Im{ a_2}+\Re{ b_2}),\\
            0= & \Re{ a_1}\Re{ a_2}+\Im{ a_1}\Im{ a_2},\\
            0=& \Re{ b_1}\Re{ b_2}+\Im{ b_1}\Im{ b_2}.
        \end{aligned}\right.
    \end{equation*}
    Without loss of generality, we can assume that $a_1=1$, so that the third equation implies that $\Re{a_2}=0$, i.e., $a_2=\alpha i$ for some $\alpha \in \R$. Substituting the information from the third and fourth equations in the other two equations we obtain
    \[ \left\{\begin{array}{ccc}
        0 & = & \alpha \Im{b_1}+\Re{b_2}, \\
        0 & = & \alpha \Re{b_1}-\Im{b_2}, 
    \end{array}  \right.\]
    so, in particular,
    \[a_1b_2-a_2b_1= \alpha \Im{b_1}+\Re{b_2} -i(\alpha\Re{b_1}-\Im{b_2})=0.\]
    However, by our assumption, $f$ and $g$ are linearly independent, so $a_1b_2-a_2b_1\neq 0$, which leads to a contradiction. Therefore, $\Re{f\overline{g}}\neq 0$, so $E$ does PR, and hence SPR.
\end{proof}

\section*{Acknowledgements}
The research of E.~Garc\'ia-S\'anchez and D.~de~Hevia is partially supported by the grants PID2020-116398GB-I00, CEX2019-000904-S and CEX2023-001347-S funded by the MICIU/AEI/10. 13039/501100011033. 
E.~Garc\'ia-S\'anchez is partially supported by the grant CEX2019-000904-S-21-3 funded by MICIU/AEI/10.13039/501100011033 and by ``ESF+''. D. de Hevia benefited from an FPU Grant FPU20/03334 from the Ministerio de Ciencia, Innovación y Universidades.  \\

This work was originated as part of the supervised research project ``Open problems in stable phase retrieval'' of the ICMAT-IMAG Doc-course in functional analysis that took place during the month of June of 2023 at ICMAT (Madrid) and IMAG (Granada). The authors want to acknowledge both institutions, as well as the organizers, for making this event possible.\\

The authors are especially grateful to Mitchell Taylor, the instructor of the course, for his tutoring and mentoring role throughout the whole project. The authors also want to thank Antonio Avilés, Eugene Bilokopytov, Jesús Illescas, Timur Oikhberg, Izak Oltman, Alberto Salguero, Tomasz Szczepanski and Pedro Tradacete for their valuable discussions and suggestions. Finally, the authors thank the anonymous referee for their kind suggestions.

\end{document}